\input louC.sty

\documentclass[a4paper,draft]{amsproc}
\usepackage{amsmath}
\usepackage{mathrsfs}
\usepackage{amssymb}
\usepackage{amscd} 

\theoremstyle{plain}
 \newtheorem{thm}{\textbf{Theorem}}[section]

\theoremstyle{definition}

\theoremstyle{remark}
 
 \numberwithin{equation}{section}

\renewcommand{\leq}{\leqslant}
\renewcommand{\geq}{\geqslant}

\title{A Characterization of $\mathcal{M}_{A_1}(\mathbb{R}^n)$}

\subjclass[2010]{Primary 30}

\author[Chu]{\bfseries Cheng Chu}

\address{
Department of Mathematics \\ 
Washington University in Saint Louis  \\ 
Saint Louis, Missouri \\
USA}
\email{chengchu@math.wustl.edu}

\begin{document}

\vspace{18mm}
\setcounter{page}{1}
\thispagestyle{empty}

\begin{abstract}
We characterize the set of all measurable functions on $\RR^n$ possessing an $A_1$ majorant, denoted as $\cM_{A_1}(\RR^n)$, by certain Banach function spaces.
We prove that a function has an $A_1$ majorant if and only if it belongs to some Banach function space for which the Hardy-Littlewood maximal operator is bounded. This answers the question posted by G. Knese, J. M$^{c}$Carthy, and K. Moen.
\end{abstract}

\maketitle

\section{Introduction}  

An $A_1$ weight $w$ is a positive locally integrable function that satisfies $$Mw\leq Cw,\,\m{a.e.}$$
Here $M$ denotes the Hardy-Littlewood maximal operator: $$Mf(x)=\sup_{x\in Q}{1\over {Q}}\int_Q |f| dx.$$

We also need the following equivalent definition of $A_1$ weight(see \cite{2}):
\beq\label{d1}
w\in A_1(\RR^n)\Llra{1\over |Q|}\int_Q wdx\leq C\inf_{Q} w_1\q\forall Q\subset\RR^n
\eeq
where the infimum is the essential infimum of $w$ over the cube $Q$.

Let $\cM_{A_1}$ be the set of measurable functions possessing an $A_1$ majorant, i.e. there exist an $A_1$ weight $w$, such that $|f|\leq w$.

To characterize $\cM_{A_1}$ on $\RR^n$, we need to use the Banach function spaces. Let us recall the notion of Banach function space. A mapping $\rho$, defined on the set of non-negative $\RR^n$-measureable functions and taking values in $[0,\infty]$, is said to be a Banach function norm if it satisfies the following conditions:

\begin{eqnarray}
\label{c1}&\q&\rho(f)=0\Llra f=0\; \m{a.e.},\;\rho(af)=a\rho(f)\;\m{for}\;a>0,\;\rho(f+g)\leq\rho(f)+\rho(g);\\
\label{c2}&\q&\m{if}\;0\leq f\leq g\;\m{a.e.},\;\m{then}\;\rho(f)\leq\rho(g);\\
\label{c3}&\q&\m{if}\;f_n\uparrow f\;\m{a.e.},\;\m{then}\;\rho(f_n)\uparrow \rho(f);\\
\label{c4}&\q&\m{if}\;B\subset\RR^n\;\m{is bounded then}\;\rho(\chi_B)<\infty;\\
\label{c5}&\q&\m{if}\;B\subset\RR^n\;\m{is bounded then}\;
\end{eqnarray}
$$\int_{B}f dx\leq C_B \rho(f)$$
$\qq\q$for some constant $C_B \in(0,\infty)$.\\\\
Given a Banach function norm $\rho$, $\cX=\cX(\RR^n,\rho)$, is the collection of all measurable functions such that $\rho(|f|)<\infty$. It is a Banach function space with norm
$$||f||_\cX =\rho(|f|).$$

In \cite{1}, G. Knese, J. M$^{c}$Carthy, and K. Moen proved by using the Rubio de Francia algorithm that
$$\bigcup \{\mathcal{X}:M\in B(\mathcal{X})\}\subset\mathcal{M}_{A_1}(\mathbb{R}^n),$$where the union is over all Banach function spaces, and conjectured these two sets are actually equal. The purpose of this note is to provide a positive answer.

\section{Main Theorem}
We are now ready to prove the theorem.

\begin{thm} \label{some label} 
$\mathcal{M}_{A_1}(\mathbb{R}^n)=\bigcup \{\mathcal{X}:M\in B(\mathcal{X})\}$.
\end{thm}

\begin{proof}
One direction is showed in \cite{1}, we prove the other direction: $$\mathcal{M}_{A_1}(\mathbb{R}^n)\subset\bigcup \{\mathcal{X}:M\in B(\mathcal{X})\}.$$
Fix $f_0 \in \mathcal{M}_{A_1}$, there exist $w_0\in A_1$ with $|f_0|\leq w_0$. Let $w_1=w_0+1$, then\,$w_1\in A_1(\RR ^n)$. Define $$\rho(f)=\inf\{C: |f|\leq Cw_1\},$$
then $\rho(f_0)<\infty$.

First, we show that $\rho(f)$ is a Banach function norm, so that $\mathcal{X}_0=\{ f |\rho(f)<\infty \}$ is a Banach function space. We verify the conditons \eqref{c1} through \eqref{c5}.

It is obvious that $\rho$ satisfies \eqref{c1} and \eqref{c2}.
Also, by the monotone convergence theorem, $\rho$ satisfies \eqref{c3}. And $|\chi_B|=1\leq w_1$
implies $\rho(\chi_B)\leq 1$, which gives the condition \eqref{c4}.
To see it satisfies \eqref{c5} let $B\in\RR^n$ be a bounded set. Suppose
\beq \label{t1}
f\leq ({1\over C_B} \int_B fdx) w_1
\eeq for some non-negative measurable function $f$, we may assume $\int_B fdx\neq 0$.
Then$$\int_B fdx\leq({1\over C_B}\int_B fdx)\int_B w_1dx$$
Thus$$\int_B w_1dx\geq C_B$$
Choose a cube $Q=Q(B)$ containing $B$, by \eqref{d1}:
$${1\over |Q|}\int_Q w_1dx\leq C\inf_{Q} w_1$$ for some constant $C$, so
\beq\label{t2}C_B\leq \int_B w_1dx\leq \int_Q w_1dx\leq |Q|C\inf_Q w_1<\infty\eeq
We may choose $C_B$ sufficiently large so that \eqref{t2}, thus \eqref{t1}, does not hold. By the definition of $\rho(f)$, we have
$$\rho(f)\geq {1\over C_B}\int_B fdx$$ which is \eqref{c5}.

Remains to show $M\in B(\mathcal{X}_0)$.
Let $f\in\cX_0$. For any $\Ge>0$, $|f|\leq (\rho(f)+\Ge)w_1$. So
$$Mf\leq (\rho(f)+\Ge)Mw_1\leq (\rho(f)+\Ge)Cw_1$$ for some constant $C$. By the definition of $\rho$,
$$\rho(Mf)\leq C(\rho(f)+\Ge)$$ Thus $M$ is bounded by $C$ on $\cX_0$.

\end{proof}

\bibliographystyle{amsplain}

\end{document}